\documentclass{llncs}
\usepackage{llncsdoc}
\usepackage{amsmath, amssymb}%
\usepackage{amsfonts}%
\usepackage{amssymb}%
\usepackage{amstext}
\usepackage{graphicx}
\usepackage{tikz}
\usetikzlibrary{patterns}
\begin{document}

\title {Realizability Models Separating Various Fan Theorems}

\author{Robert S. Lubarsky\thanks{I would like to thank Wouter Stekelenburg, Thomas
Streicher, and Jaap van Oosten for useful discussion during the
development of this work.}\inst{1} \and Michael Rathjen\inst{2}}
\institute{Dept. of Mathematical Sciences \\ Florida Atlantic
University \\ Boca Raton, FL 33431, USA
\\ \email{Robert.Lubarsky@alum.mit.edu} \\ \and Dept. of Pure Mathematics \\
University of Leeds \\ Leeds LS2 9JT, England
\\ \email{rathjen@maths.leeds.ac.uk}} \maketitle
\begin{abstract}
We develop a realizability model in which the realizers are the
reals not just Turing computable in a fixed real but rather the
reals in a countable ideal of Turing degrees. This is then applied
to prove several separation results involving variants of the Fan
Theorem.
\\{\bf AMS 2010 MSC:} 03F50, 03F60, 03D80, 03C90
\end{abstract}

\section{Introduction}

Certain constructions in computability theory lend themselves well
to realizability, the latter being based on an abstract notion of
computation. A coarse example of this is the notion of a Turing
computable function itself, as the collection of Turing machines
makes an applicative structure and so provides an example of
realizability. This model is closely tied to Turing computation,
naturally enough, and so provides finer examples. Consider Weak
K\"onig's Lemma, WKL, which is among constructivists more commonly
studied in its contrapositive form, the Fan Theorem FAN. (For
background on realizability, the Fan Theorem, and constructive
mathematics in general, there are any of a number of standard
texts, such as \cite {Bees,TvD,vO}.) Kleene's well-known eponymous
tree is a computable, infinite tree of binary sequences with no
computable path. In the context of reverse mathematics, this shows
that RCA$_0$ does not prove WKL. Within the realizability model,
the same example shows IZF does not prove FAN.

Perhaps a word should be said on the choice of the ground theory.
For the classical theory, it's RCA$_0$, while for the
constructive, it's IZF. The former is notably weak, the latter
strong. Why is that? And why those theories in particular? This is
not the place to discuss the particular choice of RCA$_0$. As for
IZF, its use is of secondary importance. The point is that much of
reverse classical mathematics is to show the equivalence of
various principles, for which a weaker base theory provides a
stronger theorem. Even for independence results, which on the
surface would be better over a stronger base, are often of the
form that a weaker theory does not imply a stronger one, where the
latter easily implies the former; clearly here, the base theory is
the weaker of the two. Also for those cases where the independence
result desired is an incomparability, the principles in question
are all weak set existence principles, weak in the sense that they
use a tiny fragment of ZF, and hence a weak base theory is needed,
to keep the theories in question from being outright provable. In
contrast, reverse constructive mathematics studies not set
existence principles, but rather logical principles. Instead of
fragments of ZF, the subject is fragments of Excluded Middle.
Especially when discussing independence, when a stronger base
theory gives a stronger theorem, in order to highlight that it
really is the logic that is up for grabs, and not set existence,
strong principles of set theory are taken as the base theory. IZF
is used here, since it is the simplest and most common
constructive correlate to ZF, the classical standard. Even for
equivalence theorems, there would still be a tendency to work over
IZF, since the degree to which a result depended on the IZF axioms
is the degree to which the result is ultimately classical. In
practice, if any theorem needs less than full IZF, what is
actually used could be read off from the proof anyway.

Returning to realizability, the picture is not quite so rosy when
it comes to other constructions. A case in point is the
distinction between WKL and WWKL, Weak Weak K\"onig's Lemma. WWKL
states that for every binary tree (of finite, 0-1 sequences) with
no path there is a natural number $n$ such that at least half the
sequences of length $n$ are not in the tree. This principle has
been studied in reverse mathematics, both classical \cite{SS} and
constructive \cite{N}. Yu and Simpson \cite{YS} showed that WWKL
does not imply WKL (over RCA$_0$). That's not so simple as merely
taking the computable sets; while that would falsify WKL, as
discussed above, it invalidates WWKL too. What they do is to
extend the computable sets by a carefully selected real $R_0$
(implicitly closing under Turing reducibility) which provides a
path through all the ``bad" trees while not destroying the Kleene
tree counter-example. That's not enough, though, because the
construction of a counter-example to WWKL relativizes. So while
$R_0$ kills off the bad computable trees, it introduces new bad
trees of its own. Hence the construction must be iterated: $R_1,
R_2, ...$ In the end, the union of the (reals computable in the)
$R_n$'s suffices.

The statement of WWKL carries over just fine to a constructive
setting, where we will call it the Weak Fan Theorem W-FAN, as well
as the question of whether W-FAN implies FAN. For better or worse,
the construction though doesn't. One might first think to use the
Yu-Simpson set of reals as the set of realizers. An immediate
problem is that we need an applicative structure: the realizers
need to act on themselves. It is immediate and routine to view
these reals as functions from the naturals to the naturals --
that's a trivial identification these days. It doesn't help
though. If the realizers are those functions, well, those
functions act on naturals, not on functions. What we would need
would be integer codes for those functions. The realizers from
this point of view would be that set of naturals, which as need be
could be taken to be functions. But here's the problem: what code
do you give a function which showed up in some increasing tower?
If you were looking at those functions computable in some fixed
oracle, you could consider all the naturals as each coding such a
machine. If instead your oracle is continually changing, it's not
clear what to do and still maintain an applicative structure. The
fixes we tried did not work, as we were warned.

The same issue comes up with another distinction around FAN,
namely the distinction between FAN and WKL. In order to show their
inequivalence, one might want to come up with a model of FAN
falsifying WKL. This has already been done using K$_2$
realizability, Kleene's notion of functional realizability. It is
another matter to prove this theorem via K$_1$ realizability. The
problem is as above: every Turing degree has an infinite binary
tree with no path of the same or lesser degree (the Kleene tree
relativized to that degree). One could take a set of degrees, any
of the trees of which have a path in some other degree. The
problem here is how to turn that into a realizability structure.

The goal of this paper is a realizability model in which the
realizers code functions from the natural numbers to themselves
with no highest Turing degree among them. As corollaries to this
method we get the two results just cited.

\section {The Main Construction}

The main idea here is to build a Kripke model, and then within
that a realizability model, which has sometimes been called
relative realizability \cite {vO}. This kind of construction was
apparently first suggested by de Jongh \cite{dJ}. Variants have
been used by several people: having the Kripke partial order
consist of only two points and the realizers be at $\bot$ certain
computable objects which are then injected into a full set of
realizers \cite{AB,B}, or using instead of Kripke semantics either
double-negation \cite{G} or a kind of Beth \cite{vO'} semantics.
For more detail on all of this, see the last two sections of
\cite{vO}.

To help keep things simple, we assume ZFC in the meta-theory. For
most of this work, neither classical logic nor the Axiom of Choice
is necessary, but we will not be careful about this.

Let the underlying Kripke partial order be $\omega^{< \omega}$.
Let M be the full Kripke model built on that p.o. Intuitively,
that means throw in all possible sets. More formally, a set in the
model is a function $f$ from $\omega^{< \omega}$ to the sets of
the model (inductively) which is non-decreasing (i.e. if $\sigma
\subseteq \tau$ then $f(\sigma) \subseteq f(\tau)$). Equality and
membership are defined by a mutual induction. On general
principles, M $\models$ IZF. Moreover, the ground model V has a
canonical image in M: given $x \in V$, let $\check{x}$ be such
that $\check{x}(\sigma) = \{\check{y} \mid y \in x\}$. We often
identify $x$ with $\check{x}$, the context hopefully making clear
whether we're in V or in M. For slightly more detail on the full
model, defined over any partial order, see for instance \cite{L}.

Within M, we will identify a special set $R$ of natural numbers,
based on a prior sequence $R_n$ of reals ($n \in \omega$). We
assume the $R_n$ are of strictly increasing Turing degree. At node
$\bot = \langle \rangle$, $R$ looks empty: $\bot \not \models s
\in R$; equivalently, $R(\langle \rangle) = \emptyset$. Suppose
inductively $R(\sigma)$ is defined, where $\sigma$ has length $n$.
Let $R_n^i$ list all the reals that differ from $R_n$ in finitely
many places. Let $R(\sigma^\frown i)$ be $R(\sigma) \cup \{
\langle n, s \rangle \mid s \in R_n^i \}$. In words, at level
$n+1$, beneath each node on level $n$, put into the $n^{th}$ slice
of $R$ all of the finite variations of $R_n$, spread out among all
the successors. So $R$ is a kind of join of the $R_n$'s, just not
all at once.

Because $R$ is (in M) a real, it makes sense to use $R$ as an
oracle for Turing computation. At $\bot$, if a computation makes
any query $s \in R$ of $R$, there are some nodes at which $s$ is
in $R$ and others where it is not, so the oracle cannot answer and
the computation cannot continue. This follows from the formal
model of oracle computability: a run of an oracle machine is a
tuple of natural numbers coding a correct computation; the rule
for extending a tuple when the last entry is an oracle call is
that the next entry must contain the right answer; if there is, at
a node, no right answer, then there can be no extending tuple.
Hence the only convergent computations are those that make no
oracle calls, and the only $R$-computable functions are the
computable ones. More generally, at node $\sigma$ of length $n$,
any query of the form $\langle k, s \rangle \in R$ with $k \geq n$
will be true at some future nodes and false at others, hence
unanswerable at $\sigma$. The computable functions at $\sigma$ are
those computable in $R_{n-1}$.

In M, let App be the applicative structure of the indices of
functions computable in $R$ (using, of course, the standard way of
turning such indices into an applicative structure). In M, let
M[App] be the induced realizability model. On general principles,
M[App] $\models$ IZF. The natural numbers of M[App] can be
identified with those of M, so any set of such in M[App] can be
identified with one in M. Furthermore, at any node, a decidable
real in one structure corresponds to a decidable real in the
other, and that can be identified with a real in the ambient
classical universe. Henceforth these various reals will not even
be distinguished notationally. For instance, if $\sigma \models
``M[App] \models ``T \subseteq \omega$ is decidable" " then we
might refer to the real $T$ in V.

For notational convenience, we will abbreviate $\sigma \models
``M[App] \models \phi"$ as $\sigma \models_{App} \phi.$

\begin {lemma}
For $\sigma$ of length $n$ and $R$ a real, $\sigma \models_{App}
``X$ is decidable" iff $X$ is Turing computable in $R_{n-1}$.
\end {lemma}

\begin {proof}
The statement ``$X$ is decidable" is $\forall m \in \omega \; m
\in X \vee m \not \in X.$ A realizer $r$ of the latter would be a
function that, on input $m$, decides whether $m$ is in or out of
$X$. If in the course of its computation $r$ asked the oracle any
question of the form $\langle k,s \rangle$ with $k \geq n$ then
the computation would not terminate at $\sigma$. So $r$ can access
only $R_{n-1}$, making $X$ computable in $R_n$. The converse is
immediate.
\end {proof}

\begin {lemma}
If $\sigma \models_{App} ``X$ is an infinite branch through the
binary tree" then $\sigma \models_{App} ``X$ is decidable."
\end {lemma}

\begin {proof}
To be an infinite branch means for every natural number $m$ there
is a unique node of length $m$. The realizer that $X$ is an
infinite branch has to produce that node given $m$.
\end {proof}

Often people are concerned about the use of various choice
principles. The independence results presented here are that much
stronger because Dependent Choice holds in our models.

\begin {proposition}
$\langle \rangle \models_{App}$ DC
\end {proposition}

\begin {proof}
The same proof that DC holds in standard Kleene K$_1$
realizability works here.
\end {proof}

\section {D-FAN and W-D-FAN}

When adapting the classical results to the current setting, we
need an additional stipulation. All of the trees, and hence
principles, we consider will be decidable: for all binary
sequences $b$ and trees $T$, either $b \in T$ or $b \not \in T$.
So, for instance, instead of the full Fan Theorem FAN, we will be
considering the Decidable Fan Theorem: if a decidable tree in
\{0,1\}$^\mathbb{N}$ has no infinite path, then the tree is
finite. Also, Weak FAN, also known as WWKL, when applied to
decidable trees, would read: if a decidable tree in
\{0,1\}$^\mathbb{N}$ has no infinite path, then there is an $n$
such that at least half of \{0,1\}$^n$ is not in the tree.

This brings us to an annoying point about notation. Decidable FAN
has been referred to in various places as D-FAN, FAN$_D$,
$\Delta$-FAN, and FAN$_\Delta$. So notation for Weak Decidable FAN
could be any of those, with a ``W" stuck in somewhere. To make
matters worse, even though the statement of Weak FAN is a
weakening of FAN and not of WKL, the name WWKL for it is already
established in the classical literature, and so one could make a
case to stick with it, and insert decidability (``D" or
``$\Delta$") somewhere in there. These same considerations apply
to other variants of FAN, whether already identified (c-FAN,
$\Pi^0_1$-FAN) or not. Whatever we do here will likely not settle
the matter. Still, we have to choose something. It strikes us as
confusing to distinguish between FAN and WKL, and then call a
variant of FAN by a variant of WKL. Also, what if somebody someday
wants to study the contrapositive of ``WWKL"? Hence we stick with
the name W-FAN. As for how to get in the decidability part, we
choose the option that's the easiest to type: W-D-FAN.

Returning to the matter at hand, Yu-Simpson \cite {YS} construct a
sequence $X^n$ of reals of increasing Turing degree such that:

$i)$ if $T$ is a tree computable in $X^n$ the branches through
which form a set of positive measure, then a path through $T$ is
computable in $X^{n+1}$, and

$ii)$ no path through the Kleene tree is computable in any $X^n$.

We apply the construction from the previous section, with $R_n$
set to $X^n$. From this, the following lemmas are pretty much
immediate.

\begin {lemma}
$\langle \rangle \models_{App}$ W-D-FAN
\end {lemma}

\begin {proof}
At any node, a decidable tree $T$ is computable in some $X^n$. If
in V the measure of the branches through $T$ were positive, then
there would be a branch computable through $X^{n+1}$. So no node
could force that there are no branches through $T$. To compute a
level at which half the nodes are not in $T$, just go through $T$
level by level until this is found.
\end {proof}

\begin {lemma}
$\langle \rangle \models_{App} \neg$D-FAN
\end {lemma}

\begin {proof}
The Kleene tree provides a counter-example.
\end {proof}

While we expect that even full W-FAN does not imply D-FAN, this
model does not satisfy W-FAN. To see that, recall that any path
through the binary tree is decidable, hence computable in some
$X^n$. There are only countably many such paths. It is easy in V
to construct a tree avoiding those countably many paths with
measure (of the paths) being as close to 1 as you'd like. The
internalization of such a tree in M[App] will not be decidable,
but will be internally a tree with no paths.

\section {FAN and WKL}

The distinction between FAN and WKL is a strange case. Their
relation is that WKL implies FAN, but not the converse. With some
exaggeration, it seems as though everyone knows that but no one
has proven it. \footnote{Thanks are due here to Hannes Diener for
first pointing this out to us and Thomas Streicher for useful
discussion.} At the very introduction of non-classical logic,
Brouwer himself must have realized this distinction, as he made a
conscious choice which variant of this class of principles he
accepted. Moreover, while he accepted FAN (having proven it from
Bar Induction), it is easy to see that WKL implies LLPO, which
Brouwer rejected. So while Brouwer did not provide what we would
today consider a model of FAN + $\neg$WKL, we would like to honor
him in the style of the Pythagoreans by attributing this result to
him, whatever may actually have been going through his mind.

Such models have since been provided, for instance by Kleene,
using his functional realizability K$_2$ \cite{KV,R}. However, in
neither of those sources is it mentioned that WKL fails. In \cite
{jB}, both FAN and WKL are analyzed into constituent principles,
it is shown that WKL's components imply the corresponding FAN
components, and it is nowhere stated that the converse does not
hold. In \cite {BI}, equivalents are given for what is there
called FAN and WKL, although their FAN is actually D-FAN, and it
is at least asked how much stronger WKL is than FAN. The one proof
we have been able to find of some fan theorem not implying WKL is
in \cite {M}, where once again the fan principle used is D-FAN.
For what it's worth, that argument, like ours, uses relative
realizability \cite {BvO}, albeit with K$_2$ realizability.

Below we give a full proof that FAN does not imply WKL. We would
be interested in hearing of other extant proofs of such, and would
find it amusing if there were none. What might be new here, if
anything, is not the result itself, but rather the methodological
point that this model is based on K$_1$. That is, while K$_1$ is
usually used to falsify even D-FAN, its variant below validates
full FAN.

\begin {theorem} (Brouwer) FAN does not imply WKL.
\end {theorem}

\begin {proof}
Take a countable $\omega$-standard model of WKL$_0$ (see
\cite{SS}, ch. VIII). There is a sequence $X^n$ of reals of
increasing Turing degree such that the reals in this model are
exactly those computable in some $X^n$. This induces a model
M[App] as in section 2 above (with $R_n$ being set to $X^n$).

To see that WKL fails, suppose to the contrary $\sigma
\models_{App} ``n \Vdash_r WKL"$. So if $\sigma \models_{App} ``e
\Vdash_r T$ is an infinite binary tree" then, at $\sigma, \;
\{n\}(e)$ must compute a path through $T$. But a path through a
tree is not computable in the tree, as is standard, by considering
the Kleene tree. This shows moreover that WKL for decidable trees
fails.

To show FAN, suppose at some node $\sigma$ in the Kripke partial
order $r$ realizes that $B$ is a bar. We must show how to compute
a uniform bound on $B$. To do this, we will build a decidable
subset $C$ of $B$. Inductively at stage $n$, suppose we have
decided $C$ on all binary sequences of length less than $n$.
Consider each binary sequence $\bar{b}$ of length $n$ (except if
$\bar{b}$ extends something in $C$ of shorter length, in which
case what happens with $\bar{b}$ just doesn't matter anymore).
Consider the path $P_{\bar{b}}$ which passes through $\bar{b}$ and
is always 0 after that. Applying $r$ to $P_{\bar{b}}$ produces a
sequence $b^+$ in $B$ on $P_{\bar{b}}$. If $b^+$ has length at
most $n$, include in $C$ all extensions of $b^+$ of length $n$,
else just include $b^+$ in $C$. After doing this for all $\bar{b}$
of length $n$, anything of length $n$ not put into $C$ is out of
$C$. This procedure terminates only when we have a uniform bound
on $C$, hence on $B$. If this never terminates, we have a
decidable, infinite set of binary sequences not in $C$ computable
from $r$. Hence at any child $\tau$ of $\sigma$ there will be an
infinite path $P$ avoiding $C$. Applying $r$ to $P$ produces an
initial segment of $P$ in $B$, say $b$. This computation itself
used only an initial segment of $P$, say $c$. Letting $n$ be the
maximum length of $b$ and $c$, at stage $n$ no initial segment of
either has been put into $C$, by the choice of $P$. So the
procedure would consider the path through $b$ and $c$ which is all
0s afterwards. This would then put the longer of $b$ and $c$ into
$C$, contradicting the choice of $P$. So this procedure must
terminate, producing a bound for $B$.

\end {proof}

\section {Questions}

1. The second construction was developed for an entirely different
purpose.

One way of stating FAN is that every bar is uniform. Weaker
versions of FAN can be developed by restricting the bars to which
the assertion applies. For instance, Decidable FAN, D-FAN, states
that for every decidable set B (i.e. for all $b$ either $b \in B$
or $b \not \in B$), if B is a bar then B is uniform.
Constructively, decidability is a very strong property; in fact,
it is the strongest hypothesis on a bar yet to be identified.
D-FAN is trivially implied by FAN; it has long been known that
D-FAN is not provable in IZF (via the Kleene tree, described in
any standard reference, such as \cite{Bees,TvD}; see \cite{LD} for
a different proof). A somewhat milder restriction on a bar B is
that it be the intersection of countably many decidable sets; that
is, B is $\Pi^0_1$ definable. Between decidable and $\Pi^0_1$ bars
are c-bars: if there is a decidable set B$'$ such that $b \in B$
iff for every $c$ extending $b \; c \in B'$, then $B$ is called a
c-set, and if it's a bar to boot then it's called a c-bar. Often
this definition seems at first unnatural and rather technical. All
that matters at the moment is that this is a weaker condition than
decidability: every decidable bar is a c-bar. c-FAN is the
assertion that every c-bar is uniform. Such principles occur
naturally in reverse constructive mathematics \cite{JR,jB06,DL},
and are all inequivalent \cite{LD}.

The first proof that D-FAN does not imply c-FAN, by Josef Berger
\cite{jB09}, went as follows. Classically, for X any collection of
bars, X-FAN and X-WKL are equivalent (as contrapositives).
Furthermore, the Turing jump of a real R can be coded into a tree
computable in R, so that c-WKL implies that jumps always exist.
Hence over RCA$_0$ c-WKL implies ACA. D-FAN and WKL (which, in the
setting of limited comprehension, is just D-WKL) are equivalent.
So if D-FAN implied c-FAN, constructively or classically, then WKL
would imply ACA, which is known not to be the case \cite{SS}.

A limitation of this argument is that it works over a very weak
base theory. It leaves open the question of whether D-FAN implies
c-FAN over IZF. While this has been settled \cite{LD}, a question
of method still remains open. Could Berger's argument be re-cast
to provide an independence results over IZF? The obvious place to
look seemed to be a model of WKL + $\neg$ACA, using the functions
there, which necessarily have no largest Turing degree, as
realizers. Our analysis of such a model did not achieve that goal.
Is there another way of turning such a model into a separation of
D- and c-FAN?

More generally, could there be any realizability model separating
D- and c-FAN? All of the realizability models we know about either
falsify D-FAN or satisfy full FAN. Perhaps that's because of the
difficulty of realizing that something is a bar. That is, to
falsify any version of the FAN, one might well want to provide a
counter-example, which would be a non-uniform bar. If a bar is not
uniform, realizing the non-uniformity would typically be trivial,
as nothing could realize that it is uniform, which suffices.
Realizing that a set is a bar is different: given a binary path,
you'd have to compute a place on that path and realize that that
location is in the alleged bar. If this set is decidable, that's
easy: continue along the path, checking each node on the way,
until you're in it. If the set cannot be assumed decidable, it is
unclear to us how to realize that it's a bar. This is something we
would like to see: a way of realizing a non-decidable set being a
bar.

2. The differences among D-FAN, c-FAN, $\Pi^0_1$-FAN, and FAN have
to do with the hypothesis; they apply to different kinds of bars.
In contrast, the difference between FAN and W-FAN has to do with
the conclusion, with whether the bar is uniform or half-uniform,
to coin a phrase. So these variants can be mixed and matched.
There are D-FAN, c-FAN, $\Pi^0_1$-FAN, FAN, and also W-D-FAN,
W-c-FAN, W-$\Pi^0_1$-FAN, and W-FAN. Clearly any version of FAN
implies its weak cousin. Other than that, we conjecture there is
complete independence between the variants of FAN and the variants
of Weak FAN. This is, we conjecture W-FAN does not imply D-FAN,
and conjoined with D-FAN does not imply c-FAN, and so on.
Furthermore, we expect that D-FAN, while of course implying
W-D-FAN, does not imply W-c-FAN, and c-FAN does not imply W-FAN,
and so on for other variants that might appear.


\begin {thebibliography} {99}
\bibitem  {AB}
    Awodey, S., Birkedal, L.:
    Elementary axioms for local maps of toposes.
    Journal of Pure and Applied Algebra
    {\bf 177(3)} (2003)
    215-230

\bibitem{B}
        Bauer, A.:
       The Realizability approach to computable analysis and
       topology.
        Ph.D. Thesis (2000)

\bibitem{Bees} Beeson, M.:
    Foundations of Constructive Mathematics, Springer-Verlag
    (1985)

\bibitem{jB06}Berger, J.:
       The Logical strength of the uniform continuity theorem.
   In Logical Approaches to Computational Barriers (eds.
      Beckmann, A.,
      Berger, U.,
      L{\"{o}}we, B., and
      Tucker, J.~V.),
      Lecture Notes in Computer Sciences,
   Springer Berlin / Heidelberg,
       (2006) 35-39

\bibitem{jB}Berger, J.:
    A decomposition of Brouwer's fan theorem.
    Journal of Logic and Analysis,
    {\bf 1(6)} (2009)
    1-8


\bibitem{jB09}Berger, J.:
       A separation result for varieties of Brouwer's fan theorem.
       In Proceedings of the 10th Asian Logic Conference (alc 10),
    Kobe University in Kobe, hyogo, Japan, September 1-6, 2008
      (eds. Arai et al.),
    World Scientific, (2010)
       85-92


\bibitem{BI}Berger, J., Ishihara, H.:
    Brouwer's fan theorem and unique existence in constructive
    analysis.
    Mathematical Logic Quarterly,
    {\bf 51(4)} (2005) 360-364,
    doi:10.1002/malq.200410038


\bibitem{BvO}  Birkedal, L., van Oosten, J.:
    Relative and modified relative realizability.
    Annals of Pure and Applied Logic
    {\bf 118(1-2)} (2002) 115-132

\bibitem{dB85}  Bishop, E., Bridges, D.:
    Constructive Analysis, Springer-Verlag (1985)

\bibitem{dJ} de Jongh, D.H.J.:
       The Maximality of the intuitionistic predicate calculus with respect to Heyting's
       Arithmetic.
        typed manuscript (1969)

\bibitem{hD08b} Diener, H.:
       Compactness under constructive scrutiny.
        Ph.D. Thesis (2008)

\bibitem{DL} Diener, H., Loeb, I.:
    Sequences of real functions on [0, 1] in constructive
    reverse mathematics.
    Annals of Pure and Applied Logic {\bf 157(1)} (2009)
    50-61

\bibitem{mF79} Fourman, M., Hyland, J.M.E.: Sheaf models for
analysis. In Applications of Sheaves (eds. Fourman, M., Mulvey,
C., Scott, D.), Lecture Notes in Mathematics {\bf 753}
    Springer Berlin / Heidelberg,
    (1979) 280-301, http://dx.doi.org/10.1007/BFb0061823

\bibitem{G} N.D. Goodman, N.D.: Relativized realizability in intuitionistic arithmetic of all finite
types.
    Journal of Symbolic Logic {\bf 43} (1978) 23-44

\bibitem{JR} Julian, W., Richman, F.: A Uniformly continuous function on [0,1]  that is everywhere
    different from its infimum.
    Pacific Journal of Mathematics (1984) {\bf 111(2)} 333-340

\bibitem{KV} Kleene, S.C., Vesley, R.E.: The Foundations of Intuitionistic Mathematics,
North-Holland (1965)

\bibitem{L} Lubarsky, R.: Independence results around Constructive
ZF.
    Annals of Pure and Applied Logic {\bf 132(2-3)} (2005) 209-225

\bibitem{LD} Lubarsky, R., Diener, H.:
    Separating the Fan Theorem and Its Weakenings.
    In Proceedings of LFCS '13 (eds. Artemov, S.N., Nerode, A.), Lecture Notes in Computer
    Science {\bf 7734}, Springer (2013) 280-295

\bibitem{M} Moschovakis, J.: Another Unique Weak K\"onig's Lemma.
In Logic, Construction, Computation (eds. Berger, U., Diener, H.,
Schuster, P., Seisenberger, M.), Ontos (2013), to appear

\bibitem{N} Nemoto, T.: Weak weak K\"onig's Lemma in constructive reverse
mathematics.
        In Proceedings of the 10th Asian Logic Conference (alc 10),
    Kobe University in Kobe, hyogo, Japan, September 1-6, 2008
      (eds. Arai et al.),
    World Scientific (2010)
       263-270


\bibitem{R} Rathjen, M.: Constructive Set Theory and Brouwerian
Principles.
    Journal of Universal Computer Science
     {\bf 11} (2005) 2008-2033

\bibitem{SS} Simpson, S.: Subsystems of Second Order Arithmetic.
    ASL/Cambridge University Press (2009)

\bibitem{TvD} Troelstra, A.S., van Dalen, D.: Constructivism in Mathematics, Vol.
1.    North Holland  (1988)

\bibitem{vO'} van Oosten, J.:
    A Semantical proof of De Jongh's theorem.
    Archive for Mathematical Logic
    {\bf 31} (1991) 105-114

\bibitem{vO}van Oosten, J.:
    Realizability: An Introduction to its Categorical Side.
    Elsevier (2008)

\bibitem{YS} Yu, X., Simpson, S.:
     Measure theory and weak K\"onig's lemma.
     Archive for Mathematical Logic
     {\bf 30} (1990) 171-180


\end {thebibliography}
\end{document}